\documentclass[11 pt]{amsart}
\usepackage{graphicx} 
\usepackage{amsmath}
\usepackage{mathtools}
\usepackage
{amstext, amscd, setspace, tikz-cd,mathtools, enumerate, graphics, latexsym, siunitx}
\usepackage{pst-node}
\usepackage{tikz-cd} 
\usepackage{amssymb}
\usepackage{hyperref}

\usepackage{PTSerif} 
\usepackage{graphicx}
\usepackage[cmtip,all]{xy}

\newcommand{\Q}{{\mathbb Q}}

\newcommand{\R}{{\mathbb R}}

\newcommand{\Z}{{\mathbb Z}}
\newcommand{\N}{{\mathbb N}}
\newcommand{\C}{{\mathbb C}}

\newcommand{\Hom}{{\rm Hom}}
\newcommand{\End}{{\rm End}}

\newcommand{\Tr}{{\rm Trace~}}
\newcommand{\dt}{{\rm det~}}

\input{xypic}
\xyoption{all}
\newtheorem{theorem}{Theorem}[section]
\newtheorem{corollary}{Corollary}[section]

\newtheorem{lemma}[theorem]{Lemma}

\newtheorem{proposition}[theorem]{Proposition}

\theoremstyle{definition}
\newtheorem{definition}[theorem]{Definition}

\newtheorem{remark}{Remark}[section]
\numberwithin{equation}{section}

\title[Linear dependence of discrete series multiplicities]{A finite Linear Dependence of Discrete Series Multiplicities}

\author{Kaustabh Mondal and Gunja Sachdeva}

\date{\today}


\address{Indian Institute of Science Education and Research, Pune, Dr.\,Homi Bhabha Road, Pashan, Pune 411008,  INDIA.}

\email{kaustabh.mondal@students.iiserpune.ac.in}

\address{BITS Pilani, K.K. BIRLA GOA Campus, NH 17B, Bypass Road, Zuarinagar, Sancoale, Goa 403726, India.}

\email{gunjas@goa.bits-pilani.ac.in}

\begin{document}
\begin{abstract}
    Let $G$ be the real points of a connected semisimple simply connected algebraic group over $\Q$ such that $G$ has a compact Cartan subgroup and let $\Gamma$ be a uniform lattice in $G$. Let $\widehat{G}_d$ denote the set of equivalence classes of unitary discrete series representations of $G$. We prove that for any finite subset of $\widehat{G}_d$ satisfying a certain condition, the associated finite set of discrete series multiplicities in $L^2(\Gamma \backslash G)$ determines all discrete series multiplicities in $L^2(\Gamma \backslash G)$. This allows us to obtain a refinement of the strong multiplicity one result for discrete series representations. As an application, we deduce that for two given levels, the equality of the dimensions of the spaces of cusp forms over a suitable finite set of weights implies the equality of the dimensions of the spaces of cusp forms for all weights.
\end{abstract}

\maketitle
\section{Introduction}
\subsection{Discrete Series Multiplicities}
Let $G$ be a semisimple connected Lie group whose complexification $G_{\C}$ is simply connected. We assume that $G$ contains a compact Cartan subgroup, i.e. $\text{rank} (G) = \text{rank} (K)$, where $K$ is a maximal compact subgroup in $G$. Consider $\Gamma$ to be a lattice in $G$ such that $\text{vol} (\Gamma \backslash G) < \infty$ with respect to the unique $G$-invariant Haar measure on $\Gamma \backslash G$. We have the right regular representation $R_{\Gamma}$ of $G$ on 
\begin{equation}
L^2(\Gamma \backslash G):= \{ \phi : G \rightarrow \C\ |\ \phi (\gamma x) =\phi (x)\ \text{for all}\ \gamma \in \Gamma, x \in G\ \text{and}\ \int\limits_{\Gamma \backslash G} |\phi(x)|^2 dx < \infty\}.
\end{equation}

\noindent Let $L^2_{0}(\Gamma \backslash G)$ be the maximal semisimple subrepresentation of $L^2(\Gamma \backslash G)$. Then

$$ L^2_{0}(\Gamma \backslash G)\cong \widehat{\bigoplus\limits_{\pi \in \widehat{G}}} m(\pi , \Gamma) \pi, $$

\noindent where $\widehat{G}$ is the set of equivalence classes of unitary irreducible representations of $G$, and $$m(\pi, \Gamma)=\text{dim}\ \Hom_G(\pi, L^2_{0}(\Gamma \backslash G)).$$

\noindent Since $G$ has compact Cartan subgroups, it possesses the discrete series representations. It is well-known that discrete series representations are parametrized by the $\textit{Harish-Chandra}$ parameters $\lambda$, and denoted by $\pi_{\lambda}$ (see Subsec. \ref{basicsetup}). The purpose of this article is to present a linear dependence among the multiplicities $m(\pi_{\lambda}, \Gamma)$ of the discrete series representations $\pi_{\lambda}$ of $G$. The multiplicity $m(\pi_{\lambda}, \Gamma)$ depends on $\text{vol} (\Gamma \backslash G)$. For example, if $\Gamma$ is a uniform torsion-free lattice in $G$ then the multiplicity $m(\pi_{\lambda}, \Gamma)$ is equal to $d_{\pi} \text{vol} (\Gamma \backslash G)$, where $d_{\pi}$ is the formal degree of $\pi$ (see \cite{kn} for the definition). Moreover, if $\Gamma$ contains nontrivial elliptic elements (i.e. torsion elements) then the $\Gamma$-conjugacy classes of these elliptic elements also contribute to $m(\pi_{\lambda}, \Gamma)$.

\vspace{5pt}

 Let us fix a maximal compact subgroup $K$ in $G$. By the assumption on $G$, there exists a compact Cartan subgroup $T \subset K$. Let $\mathfrak{g}_0, \mathfrak{k}_0, \mathfrak{t}_0$ be the Lie algebras of $G, K, T$ respectively, and  $\mathfrak{g}, \mathfrak{k}, \mathfrak{t}$ be their complexifications respectively. Let $\Phi=\Phi(\mathfrak{g}, \mathfrak{t})$ and $\Phi_c=\Phi(\mathfrak{k}, \mathfrak{t})$ be the root systems of $(G, T)$ and $(K, T)$ respectively such that $\Phi_c \subset \Phi$. Let $\Phi^+$ be a fixed but arbitrary choice of a positive system of roots in $\Phi$. A $\textit{Harish-Chandra}$ parameter $\lambda$ is an integral regular linear form in the complex dual space $\mathfrak{t}^*$ of $\mathfrak{t}$. Let $\mathcal{D}(G)$ denote the set of all Harish-Chandra parameters in $G$ (see Subsec. \ref{D^+(G)} for detailed definition). Every element $\lambda \in \mathcal{D}(G)$ determines a unique positive system of roots $\{ \alpha \in \Phi: \langle \lambda , \alpha \rangle > 0 \}$ which is denoted by $P^{\lambda}$. Any discrete series representation is uniquely determined by such a Harish-Chandra parameter $\lambda$ up to $W_c$-conjugate (see Thm. \ref{discreteseries} for details). Here, $W_c$ denotes the Weyl group of the root system $\Phi_c$.  
 
 \begin{definition}
\textit{Assume $\lambda_1, \lambda_2 \in \mathcal{D}(G)$ are given such that $P^{\lambda_1}=P^{\lambda_2}$. The infinite set $\mathcal{S}(\lambda_1, \lambda_2):= \{ \lambda_1 + k \lambda_2 \in \mathfrak{t}^* : k \in \N \cup \{0\} \}$ of integral linear forms is called a string with base $\lambda_1$ and direction $\lambda_2$. For each $k \in \N \cup \{0\}$, we denote $\lambda_1 +k \lambda_2$ by $\lambda_k$.} 
 \end{definition}
\noindent Note that every linear form $\lambda_k \in \mathcal{S}(\lambda_1, \lambda_2)$ is regular and therefore belongs to $\mathcal{D}(G)$. The discrete series representation determined by $\lambda_k \in \mathcal{D}(G)$ is denoted by $\pi_{\lambda_k}$. Each $\lambda_k$ determines the same positive system of roots i.e. $P^{\lambda_k}=P^{\lambda_1}=P^{\lambda_2}$. 

\vspace{5pt}

We need to impose a mild condition on $\lambda \in \mathcal{D}(G)$. We denote $$\mathcal{D}^{\star}(G) := \{ \lambda \in \mathcal{D}(G) : \langle \lambda-\delta^{\lambda} , \alpha \rangle > 0\ \text{for all}\ \alpha \in P^{\lambda} \cap \Phi_n\},$$

\noindent where $\Phi_n = \Phi \setminus \Phi_c$ and $\delta^{\lambda} : = \frac{1}{2} \sum\limits_{\alpha \in P^{\lambda}} \alpha $. From Lem. \ref{lemma2.4.3}, it can be seen that if the base or the direction of a string belongs to $\mathcal{D}^{\star}(G)$ then every linear form in the string belongs to $\mathcal{D}^{\star}(G)$.

\hspace{5pt}

 Let $G_e$ denote the set $\bigcup\limits_{g \in G} g K g^{-1}$ of elliptic elements in $G$. Note that $G_e$ does not depend on the choice of maximal compact subgroup $K$. The set $\Gamma \cap G_e$ of elliptic elements in $\Gamma$ is a union of \textit{finitely many} $\Gamma$-conjugacy classes i.e. $\Gamma \cap G_e = \bigcup\limits_i \ [y_i]$ for some choice of representatives $y_i \in \Gamma \cap G_e$. Finally, we define the maximal order of elliptic elements in $\Gamma$ to be the least common multiple of the orders of the elements $\{y_i \}_i$. Let us denote it by $N_{\Gamma}$. It is an invariant of $\Gamma$.

 \vspace{10pt}

\noindent Now, we are ready to display our first main result. 
\begin{theorem}\label{thm1}
\textit{Let $G$ be a semisimple connected Lie group whose complexification $G_{\C}$ is simply connected. Let $K$ be a maximal compact subgroup in $G$. Assume that $G$ contains a compact Cartan subgroup. Let $\Gamma$ be a uniform lattice in $G$, and $N_{\Gamma}$ be the maximal order of elliptic elements in $\Gamma$. For two Harish-Chandra parameters $\lambda_1$ and $\lambda_2$ with $P^{\lambda_1}=P^{\lambda_2}$, consider the string $\mathcal{S}(\lambda_1, \lambda_2)$ such that the base or the direction belongs to $\mathcal{D}^{\star}(G)$. If a finite set $\mathcal{A} \subset \N \cup \{0\}$ satisfies $| \mathcal{A} \cap (j+N_{\Gamma} \Z)| \geq |\Phi^+| +1$ for all $0 \leq j \leq N_{\Gamma}-1$, then for any $\lambda_{\ell} \in \mathcal{S}(\lambda_1, \lambda_2)$, the discrete series multiplicity $m(\pi_{\lambda_{\ell}} , \Gamma) = \sum\limits_{k \in \mathcal{A}} n(\ell, k)\ m(\pi_{\lambda_{k}} , \Gamma)$ where $n(\ell, k)$ are integers and independent of $\Gamma$.}
\end{theorem}

\noindent Since the integers $n(\ell, k)$ in the above linear expression do not depend on $\Gamma$, we have the following corollary: 

\begin{corollary}\label{cor1}
\textit{Assume the hypothesis on $G$ and $K$ as in Thm. \ref{thm1}. Let $\Gamma_1$ and $\Gamma_2$ be two uniform lattices in $G$, and $N= \text{l.c.m.}\ ( N_{\Gamma_1} , N_{\Gamma_2})$. For two Harish-Chandra parameters $\lambda_1$ and $\lambda_2$ with $P^{\lambda_1}=P^{\lambda_2}$, consider the string $\mathcal{S}(\lambda_1, \lambda_2)$ such that the base or the direction belongs to $\mathcal{D}^{\star}(G)$. Let us choose a finite set $\mathcal{A} \subset \N \cup \{0\}$ which satisfies $| \mathcal{A} \cap ( j + N \Z )| \geq |\Phi^+| +1$ for all $ 0 \leq j \leq N-1$. If $m(\pi_{\lambda_k}, \Gamma_1)= m(\pi_{\lambda_k}, \Gamma_2)$ for all $k \in \mathcal{A}$, then $m(\pi_{\lambda_k}, \Gamma_1)= m(\pi_{\lambda_k}, \Gamma_2)$ for all $\lambda_k \in \mathcal{S}(\lambda_1, \lambda_2)$.} 
\end{corollary}

\noindent In Subsec. \ref{exhaustion}, we obtain a generalised version of Thm. \ref{thm1} for a \textit{multi-directed string} (see Def. \ref{multi-directed}) of discrete series multiplicities. Even in this general case, it is sufficient to assume that only one direction belongs to $\mathcal{D}^{\star}(G)$.

\vspace{5pt}

\begin{remark}\label{remsl2r}
\textit{If $G$ has one dimensional Cartan subgroup $T$ i.e. dim $\mathfrak{t} =1$ (e.g. $G=\mathrm{SL}(2, \R)$), then for any Harish-Chandra parameter $\lambda \in \mathcal{D}^{\star}(G)$ the set $\{ k \lambda : k \in \Z \setminus \{0\} \}$ 
exhausts all but finitely many Harish-Chandra parameters in $\mathcal{D}^{\star}(G)$. Therefore, Thm. \ref{thm1} concludes that if $G$ has a one-dimensional Cartan subgroup, then there exists a finite subset $\mathcal{F} \subset \mathcal{D}^{\star}(G)$ 
such that for any $\mu \in \mathcal{D}^{\star}(G)$, the discrete series multiplicity $m(\pi_{\mu}, \Gamma)$ is determined by (a linear expression of) the multiplicities $\{ m(\pi_{\lambda}, \Gamma) : \lambda \in \mathcal{F}\}$.}
\end{remark}

\noindent As an application of Thm. \ref{thm1}, we obtain the following strong multiplicity one result for a string of discrete series representations. In fact, it holds for any infinite set of exceptions in the string with sufficiently small density. 

\begin{corollary}\label{cor2}
\textit{Assume the hypothesis on $G$ and $K$ as in Thm. \ref{thm1}. Let $\Gamma_1$ and $\Gamma_2$ be two uniform lattices in $G$, and $N=\ \text{l.c.m.}\ (N_{\Gamma_1}, N_{\Gamma_2})$. For two Harish-Chandra parameters $\lambda_1$ and $\lambda_2$ with $P^{\lambda_1}=P^{\lambda_2}$, consider the string $\mathcal{S}(\lambda_1, \lambda_2)$ such that the base or the direction belongs to $\mathcal{D}^{\star}(G)$. 
If the following condition holds: 
$$ \limsup\limits_{t \rightarrow \infty} \frac{| \left\{ 0 \leq k \leq t : m(\pi_{\lambda_k},\Gamma_1) \neq m(\pi_{\lambda_k},\Gamma_2) \right\}|}{t} < \frac{1}{N},$$
 then $ m(\pi_{\lambda_k},\Gamma_1)=m(\pi_{\lambda_k},\Gamma_2)~ \text{for all}~\lambda_k \in \mathcal{S}(\lambda_1, \lambda_2)$.}
\end{corollary}

\noindent In particular, from Cor. \ref{cor2} we can recover the following analogue of (\cite{BR}; Thm. 1.1) for a large class of discrete series representations:

\begin{corollary}
\textit{Assume the hypothesis on $G$ and $K$ as in Thm. \ref{thm1}. Let $\Gamma_1$ and $\Gamma_2$ be two uniform lattices in $G$. Assume that $m(\pi_{\lambda}, \Gamma_1) = m(\pi_{\lambda}, \Gamma_2)$ for all but finitely many discrete series representations $\pi_{\lambda}$ with $\lambda \in \mathcal{D}^{\star}(G)$. Then $m(\pi_{\lambda}, \Gamma_1) = m(\pi_{\lambda}, \Gamma_2)$ for all discrete series representations $\pi_{\lambda}$ with $\lambda \in \mathcal{D}^{\star}(G)$.} 
\end{corollary}




Let us discuss the subset $\mathcal{D}^{\star}(G)$ of $\mathcal{D}(G)$. The condition on $\lambda$ in the definition of $\mathcal{D^\star}(G)$ can be seen as a sufficiently regularity condition. If $\pi_{\lambda}$ is an integrable discrete series representation i.e. all the $K$-finite matrix coefficients of $\pi_{\lambda}$ are in $L^1(G)$, then $\lambda$ satisfies the Trombi-Varadarajan estimate (see \cite{TV}): 
$$ | \langle \lambda , \beta \rangle | > \frac{1}{2} \sum\limits_{\alpha \in P^{\lambda}} | \langle \alpha , \beta \rangle|\ \text{for every}\ \beta \in P_n^{\lambda}.$$
But if $\lambda$ satisfies the above estimate, then $\lambda$ belongs to $\mathcal{D}^{\star}(G)$. Therefore, Thm. \ref{thm1} holds for all integrable discrete series representations. Moreover, our results hold for infinitely many non-integrable discrete series as well. For instance, in the case of $\mathrm{SL}(2, \R)$, the set $\mathcal{D}(G) \setminus \mathcal{D}^{\star}(G)$ is finite. In this sense, in the definition of $\mathcal{D}^{\star}(G)$ the imposed condition on Harish-Chandra parameters is indeed a mild condition.

\vspace{5pt}

Let $\Gamma$ be a uniform lattice in $\mathrm{SL}(2, \R)$ arising from a maximal order of a division quaternion algebra over a totally real number field. Let $\mathcal{S}_k(\Gamma)$ be the space of classical holomorphic cusp forms on the complex upper half plane of weight $k$ and level $\Gamma$. It is well-known that $\mathcal{S}_k(\Gamma)$ provides an analytic realization of a discrete series representation of $\mathrm{SL}(2, \R)$. In Sec. \ref{sec-5}, we explicitly prove the equality of the dimension of $\mathcal{S}_k(\Gamma)$ and the multiplicities of discrete series representations occurring in $L^2(\Gamma \backslash \mathrm{SL}(2, \R))$ (see Prop. \ref{prop-dic}), and obtain a similar finite linear generation result on the dimension of $\mathcal{S}_k(\Gamma)$ (see Cor. \ref{cor-dim}). Combining this with Rem. \ref{remsl2r}, we have the following consequence of Thm. \ref{thm1}: 

\vspace{2pt}

\begin{corollary}\label{cor52}
\textit{Assume that $\Gamma_1$ and $\Gamma_2$ are two uniform lattices in $\mathrm{SL}(2, \R)$, also $N_{\Gamma_1}$ and $N_{\Gamma_2}$ are the maximal order of elliptic elements in $\Gamma_1$ and $\Gamma_2$ respectively. Let $N=l.c.m.(N_{\Gamma_1}, N_{\Gamma_2})$. Then there exists a finite subset $\mathcal{F}  \subset \N_{\geq 2}$ of cardinality $2N$ such that if dim $\mathcal{S}_k(\Gamma_1)$ = dim $\mathcal{S}_k(\Gamma_2)$ for all $k \in \mathcal{F}$, then dim $\mathcal{S}_k(\Gamma_1)$ = dim $\mathcal{S}_k(\Gamma_2)$ for all $k \in \N_{\geq 2}$.}
\end{corollary}

\begin{remark}
\textit{In Cor. \ref{cor52}, the finite set $\mathcal{F}$ of size $2N$ is not unique. See Cor. \ref{cor-dim} for the sufficient condition on the finite set $\mathcal{F}$ for which Cor. \ref{cor52} holds. } 
\end{remark}


\subsection{Literature Review}
In this subsection, we review the literature relevant to this work. The arithmetic and geometric significance of the discrete series multiplicities is well-known. The appropriate cohomological meaning of these multiplicities can be attributed to Langlands, who conjectured it to be the dimension of twisted sheaf cohomology of a certain degree over $\Gamma \backslash G / T$ (see \cite{RL}). This was proved partly by Schimd (see \cite{WS}), and was completed by Williams (see \cite{FW2}). Williams used the $L^2$-Dolbeault cohomology instead of sheaf cohomology, relaxing the compactness assumption on $\Gamma \backslash G$. The dimension of exactly one non-vanishing $L^2$-cohomology turns out to be the multiplicity of a discrete series representation. In \cite{BMS} with Bhagwat, the authors developed a refinement of the classical Matsushima-Murakami formula and recovered the dimension of the $L^2$-cohomology. To obtain a closed multiplicity formula for discrete series representation, Williams also proved the Thm. 7.47 of \cite{FW} which is the best possible improvement of the Osborne-Warner formula for integrable discrete series multiplicities (see \cite{OW}). In \cite{HP}, Hotta-Parthasarathy used Atiyah-Singer-Lefschetz fixed point formula to compute the index of a twisted Dirac operator explicitly. This along with Thm. \ref{thm4} of Williams provides the expected geometric expression of discrete series multiplicities (e.g. see Thm. \ref{thm6}).

\vspace{5pt}

The multiplicities $m(\pi, \Gamma)$ of irreducible representations $\pi$ of $G$ occurring in $L^2(\Gamma \backslash G)$ have been of great interest also in the context of spectral geometry of locally symmetric spaces, arithmeticity of lattices in Lie groups, etc. A spectral rigidity result of the multiplicities $m(\pi, \Gamma)$ was obtained by Bhagwat and Rajan (see \cite{BR}), and was refined later by Kelmer (see \cite{DK}). In \cite{LM}, Lauret and Miatello proved for compact group $G$ (and therefore for finite $\Gamma$) the existence of a certain finite set of multiplicities $m(\pi, \Gamma)$ that determines all multiplicities. Since every irreducible representation of a compact semisimple group can be thought of as a discrete series representation (as they occur in $L^2(G)$), it is tempting to seek some suitable finite set of discrete series multiplicities for non-compact groups which determines all discrete series multiplicities. 

\subsection{Methodology}
Let us sketch briefly the idea behind the proof of Thm. \ref{thm1}. Our method is based on a geometric expression for the multiplicities $m(\pi, \Gamma)$ of discrete series representations $\pi$ of a semisimple Lie group $G$. For instance, if $G$ is compact (and therefore any lattice $\Gamma$ is finite) then for any irreducible representation $\pi$ of $G$, the multiplicity $m(\pi, \Gamma)$ equals $\dfrac{1}{|\Gamma|} \sum\limits_{\gamma \in \Gamma} \Tr (\pi (\gamma))$. In the case of non-compact $G$, for any regular integral linear form $\lambda \in \mathcal{D}(G)$, we consider the Lefschetz number $\chi(\lambda, \Gamma)$ associated with a twisted Dirac operator acting on the space of smooth sections of a certain homogeneous bundle over $G/K$ (see Subsec. \ref{sec-2-2}). We require that $\lambda$ belongs to $\mathcal{D}^{\star}(G)$ which ensures that the multiplicity $m(\pi_{\lambda}, \Gamma)$ of the discrete series representation $\pi_{\lambda}$ equals the Lefschetz number $\chi(\lambda, \Gamma)$ up to a sign (see Thm. \ref{thm4}). Thanks to the work of Hotta and Parthasarathy (see Thm. \ref{thm6}), this Lefschetz number and hence the multiplicity admits a geometric expression in terms of the elliptic elements in $\Gamma$. A key tool in our approach is the formulation of a ``string'' $\mathcal{S}(\lambda_1, \lambda_2)$ of regular integral forms in $\mathcal{D}(G)$, with fixed base $\lambda_1$ and direction $\lambda_2$ (see Def. \ref{def-2-4}). For such a given string, we define a generating function $F_{\lambda_1, \lambda_2, \Gamma}$ as a formal power series in the complex variable $z$ with coefficients as discrete series multiplicities associated to the string (see Sec. \ref{sec-3}). Using the geometric expression of the multiplicities, we prove that $F_{\lambda_1, \lambda_2, \Gamma}$ is a rational function with numerator polynomial $p(z)=\sum\limits_i b_i z^i$ and denominator $(1-z^{N_{\Gamma}})^{|\Phi^+| +1}$, where $N_{\Gamma}$ is the maximal order of elliptic elements in $\Gamma$, and the degree of $p(z) < N_{\Gamma} ( |\Phi^+| + 1)$ (see Prop. \ref{propgen}). Finally, the hypothesis on the finite set $\mathcal{A} \subset \N \cup \{0\}$ gives rise to a nonsingular system of $|\Phi^+| +1$ many linear equations in the unknowns $\{ b_i \}_i$ and constant terms $\{m(\pi_{\lambda_k}, \Gamma) : k \in \mathcal{A} \}$. Solving this system uniquely determines the desired linear expression for each discrete series multiplicity.

\subsection{Organization}
The article is organized as follows: In Subsec. \ref{basicsetup}, we recall the Harish-Chandra parametrization of discrete series representations. The relation between discrete series representations and Lefschetz numbers, as well as the geometric expression of Lefschetz numbers, is described in Subsec. \ref{sec-2-2}, and Subsec. \ref{sec-2-3} respectively. The Subsec. \ref{sec-2-4} describes the notion of string of regular linear integral forms. The Sec. \ref{sec-3} is devoted to introducing the generating function for a string and proving its rationality. The proof of the main results Thm. \ref{thm1} and Cor. \ref{cor2} are given in Subsec. \ref{proofthm1} and Subsec. \ref{proofcor2} respectively. In Subsec. \ref{exhaustion}, we introduce the notion of multi-string and provide a generalized version of Thm. \ref{thm1}. Subsequently we prove the infinitude of its remainder part in $\mathcal{D^*}(G)$. In Sec. \ref{sec-5}, as an application of our results, we showcase a finite linear dependence result for the dimension of cusp forms.

\vspace{10pt}

\noindent \textbf{Acknowledgement:} K. Mondal is grateful to Dipendra Prasad for inviting him to present this work at the Indian Institute of Technology, Bombay. The authors thank him for his careful reading of the manuscript and for suggesting a correction regarding the order of elements in the $\Gamma$-conjugacy classes of elliptic elements in $\Gamma$. The authors also thank A. Raghuram, U. K. Anandavardhanan, and Chandrasheel Bhagwat for their invaluable comments on this article. K. Mondal is supported by the Prime Minister Research Fellowship (PMRF), Govt. of India. He thanks BITS Pilani Goa campus for the warm hospitality during his visit, where some parts of this work were carried out. G. Sachdeva is supported by the Department of Science and Technology–Science and Engineering Research Board, Govt. of India POWER Grant [SPG/2022/001738]. 

\section {Preliminaries}
\subsection {Basic Setup} \label{basicsetup}

Let $G$ be a connected semisimple non-compact Lie group with finite center. We fix once and for all a maximal compact subgroup $K$ in $G$. Assume that $G$ has a compact Cartan subgroup $T \subset K$; equivalently, $\text{rank} (G) = \text{rank} (K)$. Let $\mathfrak{g}_0, \mathfrak{k}_0, \mathfrak{t}_0$ be the Lie algebras of $G, K, T$ respectively, and $\mathfrak{g}, \mathfrak{k}, \mathfrak{t}$ be their complexifications respectively. Let $\Phi= \Phi(\mathfrak{g}, \mathfrak{t})$ and $\Phi_c= \Phi( \mathfrak{k} , \mathfrak{t})$ be the root systems of $(G, T)$ and $(K, T)$ respectively such that $\Phi_c \subset \Phi$, and  let $W_c \subset W$ be their corresponding Weyl groups respectively. The set $\Phi_c$ is called the set of compact roots, and its complement $\Phi_n$ in $\Phi$ is called the set of non-compact roots. Let $\mathfrak{p}_0$ be the orthogonal complement of $\mathfrak{k}_0$ in $\mathfrak{g}_0$ with respect to the killing form. Denote by $\mathfrak{g}_{\alpha}$ the root space of the root $\alpha \in \Phi$. Then $\alpha \in \Phi_c$ if and only if $\mathfrak{g}_{\alpha} \subset \mathfrak{k}$, and $\alpha \in \Phi_n$ if and only if  $\mathfrak{g}_{\alpha} \subset \mathfrak{p}$. Moreover, $\text{dim}\ \mathfrak{p}$ is even. We fix an arbitrary positive system $\Phi^+$ in $\Phi$ with $\Phi_c^+ = \Phi^+ \cap \Phi_c$, $ \Phi_n^+ = \Phi_n \cap \Phi^+$, $\delta= \dfrac{1}{2} \sum\limits_{\alpha \in \Phi^+} \alpha$, $\delta_c =  \dfrac{1}{2} \sum\limits_{\alpha \in \Phi_c^+} \alpha$, and $ \delta_n =  \dfrac{1}{2} \sum\limits_{\alpha \in \Phi_n^+} \alpha = \delta - \delta_c$.

Let $\mathcal{F}$ be the set of all linear forms $\lambda \in \mathfrak{t}^*$ such that $\dfrac{2 \langle \lambda - \delta, \alpha \rangle }{\langle \alpha , \alpha \rangle}$ is an integer. Such a linear form is called an integral linear form.

\vspace{5pt}

\noindent\label{D^+(G)} We define the set \textit{$\mathcal{D}(G) : = \{ \lambda \in \mathcal{F} : \langle \lambda , \alpha \rangle \neq 0\ \text{for all}\ \alpha \in \Phi;\ \text{and}\ \langle \lambda , \alpha \rangle > 0\ \text{for all}\ \alpha \in \Phi_c^+ \}$}. An integral linear form $\lambda$ satisfying $\langle \lambda , \alpha \rangle \neq 0$ for all $\alpha \in \Phi$ is called a \textit{regular} integral linear form. Note that in the definition of $\mathcal{D}(G)$, the condition $\langle \lambda , \alpha \rangle > 0\ \text{for all}\ \alpha \in \Phi_c^+$ ensures that no two $\lambda_1, \lambda_2 \in \mathcal{D}(G)$ are $W_c$-conjugates of each other. The following celebrated theorem of Harish-Chandra says that the set $\mathcal{D}(G)$ completely parameterizes the set of all equivalence classes of discrete series representations.

\begin{theorem}(\cite{HC})\label{discreteseries}
\textit{Let $\lambda \in \mathcal{D}(G)$. Then there is a discrete series representation $\pi_{\lambda}$ of $G$ whose global character $\Theta_{\pi_{\lambda}}$, as a locally $L^1$-function on $G$, is given on $T$ by 
$$ \Theta_{\pi_{\lambda}}( e^X ) = \frac {(-1)^m \prod\limits_{\alpha \in \Phi^+} \langle \lambda , \alpha \rangle \sum\limits_{w \in W_c} \dt(w)\ e^{w \lambda (X)}}{\prod\limits_{\alpha \in \Phi^+} ( e^{\frac{\alpha(X)}{2}} - e ^{- \frac{\alpha (X)}{2}})}$$ for $X \in \mathfrak{t}$, and $\text{dim}\ G/K= 2m$. Every discrete series representation is unitarily equivalent to $\pi_{\lambda}$ for some $\lambda \in \mathcal{D}(G)$.} 

\end{theorem} 

\begin{remark}
\textit{Since no two linear forms in $\mathcal{D}(G)$ are $W_c$-conjugates to each other, the set $\mathcal{D}(G)$ completely parametrizes the set of equivalence classes of discrete series representations of $G$.}
\end{remark}

\noindent If an integral linear form $\lambda$ is regular, then it gives rise to a positive system $P^{\lambda} =\{ \alpha \in \Phi : \langle \lambda , \alpha \rangle > 0 \}$. We denote $\delta^{\lambda} =\frac{1}{2} \sum\limits_{\alpha \in P^{\lambda}} \alpha $.

\begin{lemma}
\textit{For a given positive chamber (i.e. a given positive system) $\Sigma^+$ in $\Phi$, there exists a basis $\{ \lambda_1, \lambda_2, \dots ,\lambda_n\}$ of $\mathfrak{t}^*$ consisting of elements in $\mathcal{D}(G)$ such that $P^{\lambda_i}=\Sigma^+$.}
\end{lemma}

\begin{proof}
We know that $\{ \lambda \in \mathfrak{t}^* : \langle \lambda, \alpha \rangle > 0\ \text{for all}\ \alpha \in \Phi^+ \}= \{ \lambda \in \mathfrak{t}^* : P^{\lambda} = \Phi^+\}$ is an open subset of $\mathfrak{t}^*$. Translating by a suitable Weyl group element we have the open set $\{ \lambda \in \mathfrak{t}^* : P^{\lambda} = \Sigma^+ \}$. Now it is straightforward to see that any open set in a finite-dimensional vector space (with respect to standard Euclidean topology) contains a basis of the vector space. Therefore, we can choose a $\Z$-basis $\{ \lambda_1, \dots, \lambda_n\}$ with $\lambda_i \in \{ \lambda \in \mathfrak{t}^* : P^{\lambda}=\Sigma^+\}$.  
\end{proof}

Consider $\Gamma$ to be a uniform lattice in $G$. We have a right regular representation $R_{\Gamma}$ of $G$ on $L^2(\Gamma \backslash G)$. It decomposes as a Hilbert direct sum of unitary irreducible representations $\pi$ of $G$ with finite multiplicity; 
$$ L^2(\Gamma \backslash G) \cong \widehat{\bigoplus_{\pi \in \widehat{G}}} m(\pi, \Gamma) \pi. $$

\noindent Here $m(\pi, \Gamma)$ is the multiplicity of $\pi$ in $(R_\Gamma, L^2(\Gamma \backslash G))$; in other words, $m(\pi, \Gamma)=\text{dim}\ \Hom_G(\pi, R_{\Gamma}) $. The multiplicity of a discrete series representation $\pi_{\lambda}$ with the Harish-Chandra parameter $\lambda$ is denoted by $m(\pi_{\lambda} , \Gamma)$.

\subsection{Lefchetz number and discrete series multiplicities}\label{sec-2-2}

In this subsection, we define the Lefschetz number; and describe the relation between the Lefschetz number and the discrete series multiplicities. We need to introduce the notion of an invariant Dirac operator on $G/K$. The reader is referred to Chap. 6 of \cite{FW} for detailed description.

\vspace{5pt}

\noindent  We have $\mathfrak{p}= \sum\limits_{\alpha \in \Phi_n} \mathfrak{g}_{\alpha}$. We choose a basis $\{X_i\}$ and $\{Y_j\}$ of $\mathfrak{k}$ and $\mathfrak{p}$ respectively, such that with respect to a bilinear form $B$ on $\mathfrak{g}$ induced from the Killing form, we have $B(X_k, X_l)= - \delta_{kl}$ and $B(Y_m, Y_n)= \delta_{mn}$. Let $C(\mathfrak{p})$ be the Clifford algebra of $\mathfrak{p}$ with respect to the non-degenerate bilinear form $B$. Since $\text{dim}\ \mathfrak{p}$ is even, there is an algebra isomorphism $\epsilon$ between $C(\mathfrak{p})$ and the matrix algebra $\End(S)$ for some vector space $S$ with dimension $2^{(\frac{\text{dim}\ \mathfrak{p}}{2})}$. The restriction of $\epsilon$ to the spin group $\text{Spin}(\mathfrak{p}_0) \subset C(\mathfrak{p})$ gives a spin representation $\sigma := \epsilon|_{\text{Spin}(\mathfrak{p}_0)}$ on $S$. The spin representation $\sigma$ turns out to be the direct sum of two ($\tfrac{1}{2}$ spin) irreducible representations $\sigma^{\pm}$ on $S^{\pm}$, and $S = S^+ \bigoplus S^-$. Let us denote the Casimir element by $\Omega$, which is given by $ \Omega = - \sum\limits_i X_i^2 + \sum\limits_j Y_j^2$, in the center of the universal enveloping algebra $\mathfrak{U(g)}$ of $\mathfrak{g}$.

\vspace{5pt}

\noindent Now fix a $\lambda \in \mathcal{D}(G)$. It can be easily seen that $\lambda - \delta_c$ is $\Phi^+_c$-dominant. Therefore, there exists a finite dimensional irreducible representation $\tau_{\lambda -\delta_c}$ of $\mathfrak{k}$ on a finite-dimensional vector space $V_{\lambda - \delta_c}$ with highest weight $\lambda - \delta_c$. Let us denote $\chi_{\lambda}^{\pm} := \sigma^{\pm} \otimes \tau_{\lambda - \delta_c}$ to be a representation of $\mathfrak{k}$ (note that $\mathfrak{k}$ is also the complexified Lie algebra of the spin group $\text{Spin}\ \mathfrak{p}_0$). This may not integrate to a representation of $K$ in general. However under the assumption that $\lambda + \delta_n^{(\lambda)}$ exponentiates to a character of $T$, $\chi_{\lambda}^{\pm}$ does integrate to a representation of $K$. We consider the smooth homogeneous vector bundle $E_{\lambda}^{\pm} := G \times_K \chi_{\lambda}^{\pm}$ on $G/K$. Let $\Gamma^{\infty} E^{\pm}_{\lambda}$ be the space of smooth sections of a homogeneous vector bundle $E^{\pm}_{\lambda}$. Then there exist $1^{\text{st}}$ order $G$-invariant elliptic differential operators $D_{\lambda}^{\pm} : \Gamma^{\infty} E_{\lambda}^{\pm} \longrightarrow \Gamma^{\infty} E_{\lambda}^{\mp}$. These are Dirac operators twisted by $\tau_{\lambda - \delta_c}$. With respect to appropriate metric, $D_{\lambda}^+$ and $D_{\lambda}^-$ are adjoints of each other, and their composition $D_{\lambda}^{\pm} D_{\lambda}^{\mp}$ has the following simple description due to R. Parthasarathy:

\begin{theorem}(\cite{RP})
\textit{For all $\lambda \in \mathcal{D}(G)$, we have $D_{\lambda}^{\pm} D_{\lambda}^{\mp}= - \Omega + \langle \lambda-\delta_c , \lambda + \delta +\delta_n \rangle$, where $\Omega$ is the Casimir element.} 
\end{theorem}

\noindent Finally, we define $H^{\pm}(E_{\lambda} , \Gamma)$ to be the space of all $\Gamma$-invariant sections $s \in \Gamma^{\infty} E_{\lambda}^{\pm}$ such that $D_{\lambda}^{\pm} s = 0 $. Now we are ready to define the \textit{Lefschetz number} and state the theorem that relates it with the discrete series multiplicities.



\begin{definition}\label{lefschetzdef}
\textit{For any $\lambda \in \mathcal{D}(G)$ and a uniform lattice $\Gamma$ in $G$, the difference $\text{dim} H^+(E_{\lambda} , \Gamma) - \text{dim} H^-(E_{\lambda} , \Gamma)$ is called the Lefschetz number, and is denoted by $\chi(\lambda , \Gamma)$.}
\end{definition}

\begin{remark}
\textit{If $\Gamma$ is a uniform torsion-free lattice in $G$, then the elliptic operator $D^+_{\lambda}$ projects to a first order elliptic differential operator $D_{\Gamma, \lambda}^+$ on the locally symmetric space $\Gamma \backslash G / K$. In this case, the above Lefschetz number $\chi(\lambda , \Gamma)$ is called the index of the differential operator $D_{\Gamma, \lambda}^+$.}\end{remark}

\begin{theorem}(Thm. 3.3; \cite{FW1})\label{thm4}
\textit{Suppose that $\lambda \in \mathcal{D}^{\star}(G)$, and $\pi_{\lambda} \in \widehat{G}$ is the associated discrete series representation of $G$. Let $w_0 \in W$ be the unique Weyl group element such that $w_0 \Phi^+ = P^{\lambda}$, and let $l(w_0) = | w_0 (- \Phi^+) \cap \Phi^+ |$ be the length of $w_0$. Then for any uniform lattice $\Gamma$ in $G$, the multiplicity $m(\pi_{\lambda}, \Gamma)$ of $\pi_{\lambda}$ in $L^2(\Gamma \backslash G)$ is given by 
$$ m(\pi_{\lambda}, \Gamma) = (-1)^{l(w_0)} \{ \text{dim} H^+(E_{\lambda} , \Gamma) - \text{dim} H^-(E_{\lambda} , \Gamma) \} = (-1)^{l(w_0)} \chi(\lambda , \Gamma).$$}

\end{theorem}

\subsection{The results of Hotta-Parthasarathy on Lefschetz number}\label{sec-2-3}

In this subsection, we review a couple of results of Hotta and Parthasarathy on the Lefschetz number $\chi(\lambda , \Gamma)$ from section 2 of \cite{HP}. For any $y \in T$, let $G_y^0$ be the identity component of the centraliser $G_y$ of $y$ in $G$. Let $W_y$ be the Weyl group for the root system $\Phi(G_y^0, T)$. Let $K_y$ be the centraliser of $y$ in $K$ and $d = \text{dim}\ {G/K}, d_{y}=\text{dim}\ G_y / K_y$. We use the convention that $e^{\alpha}$ is the character of $T$ corresponding to an element $\alpha \in \mathfrak{t}^*$. Let $\Phi_y = \{ \alpha \in \Phi : e^{\alpha}(y)=1\}$. Then $\Phi_y^+=\Phi^+ \cap \Phi_y$ is a positive system of $\Phi(G_y^0, T)$.  

\vspace{5pt}

\noindent For any $\lambda \in \mathcal{D}(G)$, we define the function $\Psi_{\lambda}^0$ on $T$ as follows: 
$$\Psi_{\lambda}^0(y)= \dfrac{\sum\limits_{w \in W_c/W_y} (-1)^{l(w)} e^{w \lambda -\delta} (y) \prod\limits_{\alpha \in \Phi_y^+} \langle w \lambda , \alpha \rangle}{\prod\limits_{\alpha \in \Phi^+ \setminus \Phi_y^+} (1 - e^{-\alpha}(y))}.$$

\noindent Here $\delta=\frac{1}{2}\sum\limits_{\alpha \in \Phi^+} \alpha$, and $W_y$ is considered as a subgroup of $W_c$. The function $\Psi_{\lambda}^0$ is $W_c$-invariant in $\lambda$, and it extends to an $\text{Ad}(G)$-invariant function on the set $G_e$ of elliptic elements of $G$ i.e. the set of elements of $G$ conjugate to some element of $K$. 

\begin{remark}
\textit{If $y$ is a regular element in $G$, then $\Phi_y=\emptyset$. So we take $\prod\limits_{\alpha \in \Phi_y^+} \langle w \lambda , \alpha \rangle=1$.}
\end{remark}

\noindent Finally, for any $\lambda \in \mathcal{D}(G)$ we define the function $\Psi_{\lambda}(y)$ in $y$ as follows: 

$$\Psi_{\lambda}(y)=\dfrac{ (-1)^{\frac{d + d_y}{2}} \prod\limits_{\alpha \in \Phi_y^+} (\langle \delta_y, \alpha \rangle)^{-1}\ \Psi_{\lambda}^0(y)}{|W_y| [G_y : G_y^0] },$$


\noindent For any element $y \in \Gamma$, we denote $\Gamma_y$ as the centralizer of $y$ in $\Gamma$. Then we have the following result of Hotta and Parthasarathy:

\begin{theorem}(Thm. 3; \cite{HP})\label{thm5}
\textit{Let $\Gamma$ be a uniform lattice in $G$. For an arbitrary $\lambda \in \mathcal{D}(G)$ we have $$\chi(\lambda, \Gamma)= (-1)^{\ell(w_o)}\sum\limits_{[y]} \text{vol} (\Gamma_y \backslash G_y) \Psi_{\lambda}(y),$$ where the sum runs over the (finite) set of $\Gamma$-conjugacy classes of elliptic elements in $\Gamma$ and $w_0$ be the unique element in $W$ such that $w_0 \Phi^+ = P^{\lambda}$.} 
\end{theorem}

\begin{remark}
\textit{Note that the above theorem does not require $\lambda \in \mathcal{D}^{\star}(G)$.} 
\end{remark}

\noindent Combining Thm. \ref{thm4} and Thm. \ref{thm5}, we deduce the following:

\begin{theorem}\label{thm6}
\textit{Let $\lambda \in \mathcal{D}^{\star}(G)$ and $w_0 \in W$ such that $w_0 \Phi^+=P^{(\lambda)}$. Then $$m(\pi_{\lambda}, \Gamma) =\sum\limits_{[y]} \text{vol}(\Gamma_y \backslash G_y) \Psi_{\lambda}(y),$$ where the sum runs over the set of $\Gamma$-conjugacy classes of elliptic elements in $\Gamma$.}
\end{theorem}

\begin{corollary}\label{corvolume}
Assume the hypothesis on $\lambda \in \mathcal{D}^{\star}(G)$ as in Thm. \ref{thm6}. If $\Gamma$ is uniform torsion-free lattice in $G$, then $m(\pi_{\lambda}, \Gamma)=\ \text{vol}(\Gamma \backslash G) d_{\pi_{\lambda}}.$
\end{corollary}

\begin{proof}
The proof follows from the identity: $\Psi_{\lambda}(1)= \dfrac{\prod \limits_{\alpha \in \Phi^+} \langle \lambda , \alpha \rangle}{\prod \limits_{\alpha \in \Phi^+} \langle \delta , \alpha \rangle |W|}=d_{\pi_{\lambda}}$.
\end{proof}

\begin{remark}
\textit{The Thm. 7.47 of \cite{FW} is an extension of Thm. \ref{thm4} for the lattices $\Gamma$ which are torsion-free, of finite $\text{vol} (\Gamma \backslash G)$ and satisfy some Langlands assumptions. These latter assumptions on $\Gamma$ are automatic either when $\Gamma$ is an arithmetic lattice and $G/K$ is a rank one symmetric space or when $G$ has no compact simple factor. Since the multiplicity of discrete series representation for torsion free lattices $\Gamma$ is very explicit by Cor. \ref{corvolume}, the only interesting cases in our context are for $\Gamma$ having nontrivial elliptic elements.} 
\end{remark}

\subsection{String of regular integral linear forms}\label{sec-2-4}
In this subsection we introduce the notion of a string of regular integral linear forms in $\mathcal{D}(G)$ and $\mathcal{D}^{\star}(G)$.  
\begin{definition}\label{def-2-4}
\textit{Given two integral regular linear forms $\lambda_1 , \lambda_2 \in \mathcal{D}(G)$, the countably infinite set $\mathcal{S}(\lambda_1 , \lambda_2) := \{ \lambda_1 + k \lambda_2 : k \in \N \cup \{0\} \}$ is called a string of integral regular linear forms with base $\lambda_1$ and direction $\lambda_2$.} 
\end{definition}

For $\lambda_1 , \lambda_2 \in \mathcal{D}(G)$, let $P^{\lambda_1}$ and $P^{\lambda_2}$ be the corresponding positive root systems for the regular linear forms $\lambda_1$ and $\lambda_2$ respectively. If $P^{\lambda_1} \neq P^{\lambda_2}$, then there exists $\sigma \in W$ such that $\sigma P^{\lambda_1} =P^{\lambda_2}$. On the other hand, $\sigma P^{\lambda_1} = P^{\sigma \lambda_1}$. Therefore without loss of generality, we can consider $\lambda_1, \lambda_2 \in \mathcal{D}(G)$ such that $P^{\lambda_1} = P^{\lambda_2}$. 

\begin{lemma}
\textit{If $\lambda \in \mathcal{D}^{\star}(G)$, then $k \lambda$ belongs to $\mathcal{D}^{\star}(G)$ for every $k \in \N$.}
\end{lemma}

\begin{proof}
For any $k \in \N$, it is easy to see that $k \lambda \in \mathcal{D}(G)$ and $P^{\lambda} = P^{k \lambda}$ i.e. $\lambda$ and $k \lambda$ determine the same positive chamber. We have
\vspace{2pt}
\begin{equation}
\begin{split}
 \langle k \lambda - \delta^{k \lambda}, \alpha \rangle
& = \langle (k-1) \lambda , \alpha \rangle + \langle \lambda - \delta^{\lambda}, \alpha \rangle \\
& = (k-1) \langle \lambda , \alpha \rangle + \langle \lambda - \delta^{\lambda}, \alpha \rangle.\\ 
\end{split}
\end{equation}

\noindent Now for any $\alpha \in P^{\lambda} \cap \Phi_n$, $\langle \lambda , \alpha \rangle > 0$ and by the hypothesis, $\langle \lambda - \delta^{\lambda}, \alpha \rangle > 0$. Hence $k \lambda \in \mathcal{D}^{\star}(G)$. 
\end{proof}

Before describing the next lemma, we record a useful observation. If $\lambda_1, \lambda_2 \in \mathcal{D}(G)$, then one can easily check that $P^{(\lambda_1)} \cap P^{(\lambda_2)} \subset P^{(\lambda_1 + \lambda_2)} \subset P^{(\lambda_1)} \cup P^{(\lambda_2)}$. 

\begin{lemma}\label{lemma2.4.3}
\textit{Consider two linear forms $\lambda_1, \lambda_2 \in \mathcal{D}(G)$ with $P^{\lambda_1}=P^{\lambda_2}$. If $\lambda_1$ or $\lambda_2$ belongs to $\mathcal{D}^{\star}(G)$, then every linear form in the string $\mathcal{S}(\lambda_1, \lambda_2)$ belongs to $\mathcal{D}^{\star}(G)$.} 
\end{lemma}

\begin{proof}
Indeed, $\lambda_1 + k \lambda_2 \in \mathcal{D}(G)$ for any $k \in \N$ and $P^{(\lambda_1)} \cap P^{(k \lambda_2)} \subset P^{(\lambda_1 + k \lambda_2)} \subset P^{(\lambda_1)} \cup P^{(k \lambda_2)}$. Therefore all integral linear forms of the string determine the same positive system as the one determined by the base and the direction of the string. Without loss of generality, we assume that $\lambda_2 \in \mathcal{D}^{\star}(G)$. Let us write $P_n^{\lambda} =P^{\lambda} \cap \Phi_n$. Then for any $\alpha \in P_n^{( \lambda_1 + k \lambda_2)} = P_n^{(k \lambda_2)} = P_n^{( \lambda_1)} \subset P^{(\lambda_1)}$, we have
$$ \langle \lambda_1 + k \lambda_2 - \delta^{(\lambda_1 + k \lambda_2)}, \alpha \rangle = \langle \lambda_1 , \alpha \rangle + \langle k \lambda_2 - \delta^{(k \lambda_2)} , \alpha \rangle > 0. $$ 

\noindent Therefore, every member of the string $\mathcal{S}(\lambda_1, \lambda_2)$ belongs to $\mathcal{D}^{\star}(G)$.
\end{proof}

\begin{proposition}
\textit{Let $\pi_{\lambda}$ be the discrete series representation with Harish-Chandra parameter $\lambda \in \mathcal{D}(G)$. Assume that $\lambda_1, \lambda_2 \in \mathcal{D}(G)$ with $P^{\lambda_1}=P^{\lambda_2}$. Then $\{ \pi_{\lambda_1 + k \lambda_2} : k \in \N \}$ is an infinite family of inequivalent discrete series representations of $G$.}
\end{proposition}

\begin{proof}
We know that two discrete series representations $\pi_{\lambda_1}$ and $\pi_{\lambda_2}$ are equivalent if and only if $\lambda_1$ and $\lambda_2$ are $W_c$-conjugate. We  need to show that for two distinct positive integers $k_1$ and $k_2$, the linear forms $\lambda_1 + k_1 \lambda_2$ and $\lambda_1 + k_2 \lambda_2$ are not $W_c$-conjugate. 

\noindent If $w \in W_c$ and $\mu \in \mathcal{D}(G)$ such that $\langle \mu , \alpha \rangle > 0\ \text{for all}\ \alpha \in \Phi_c^+$, then by the definition of the length of $w$ there exists $\alpha \in \Phi_c^+$ (which depends on $w$) such that $\langle w \mu , \alpha \rangle < 0$. In other words, if $\mu_1, \mu_2 \in \mathcal{D}(G)$ with $\langle \mu_1 , \alpha \rangle > 0$ and $ \langle \mu_2, \alpha \rangle > 0\ \text{for all}\ \alpha \in \Phi_c^+$ then $\mu_1$ and $\mu_2$ are \textit{not} $W_c$-conjugate. 

\noindent Thus it is enough to show that $\langle \lambda_1 + k \lambda_2 , \alpha \rangle > 0\ \text{for all}\ \alpha \in \Phi_c^+$. For any $\alpha \in \Phi_c^+$, we have $\langle \lambda_1 + k \lambda_2, \alpha \rangle = \langle \lambda_1 , \alpha \rangle + k \langle \lambda_2, \alpha \rangle $. Now $\lambda_1, \lambda_2 \in \mathcal{D}(G)$ implies that $\Phi_c^+ \subset P^{(\lambda_2)} = P^{(k \lambda_2)} = P^{(\lambda_1)}$. Therefore, $\langle \lambda_1 , \alpha \rangle > 0$ and $\langle \lambda_2, \alpha \rangle > 0\ \text{for all}\ \alpha \in \Phi_c^+$. Hence the proof follows.  
\end{proof}

\begin{remark}
\textit{ In fact the above proposition shows that no two linear forms in a string give rise to equivalent discrete series representations. In other words, each member of a string determines a discrete series representation which is inequivalent to the rest.}
\end{remark}

\section{Generating Function for a String}\label{sec-3}

\noindent For two Harish-Chandra parameters $\lambda_1, \lambda_2$, let $\mathcal{S}( \lambda_1 , \lambda_2)$ be a string such that the base $\lambda_1$ or the direction $\lambda_2$ belongs to $\mathcal{D}^{\star}(G)$. For this string, we formally define the generating function in complex variable $z$ as 
$$F_{\lambda_1, \lambda_2, \Gamma} (z) : = \sum\limits_{k \in \N} m(\pi_k , \Gamma) z^{k-1},$$

\noindent where $\pi_k = \pi_{\lambda_1 + k \lambda_2 }$ with $k \in \N$. Let $\lambda_k = \lambda_1 + k \lambda_2$ and $w_0 \in W$ such that $w_0 \Phi^+ = P^{\lambda_k}=P^{\lambda_1}=P^{\lambda_2}$. Then by Thm. \ref{thm4}, the generating function $F_{\lambda_1, \lambda_2, \Gamma} (z)$ can be rewritten as $$ F_{\lambda_1, \lambda_2, \Gamma}(z) = \sum\limits_{k \in \N} (-1)^{l(w_0)} \chi(\lambda_k, \Gamma) z ^{k-1} = (-1)^{l(w_0)} \sum\limits_{k \in \N} \chi(\lambda_k, \Gamma) z^{k-1}.$$

\noindent The number of $\Gamma$-conjugacy classes in the set $\Gamma \cap G_e$ of elliptic elements in $\Gamma$ is finite. We fix a (finite) set of representatives $\{y_i\}$ of these classes. The least common multiple of the order of the elements $\{y_i\}_i$ is called the maximal order of the elliptic elements in $\Gamma$, and it is denoted by $N_{\Gamma}$.

\begin{proposition}\label{propgen}
\textit{Let $N_{\Gamma}$ be the maximal order of elliptic elements in $\Gamma$. Let $\mathcal{S}(\lambda_1, \lambda_2)$ be a string with $P^{\lambda_1}=P^{\lambda_2}$, and $F_{\lambda_1, \lambda_2, \Gamma}$ be its associated generating function in the complex variable $z$. Then there exists a complex polynomial $p_{\lambda_1, \lambda_2, \Gamma}$ with degree less than $N_{\Gamma} ( |\Phi^+| +1)$ such that 
$$ F_{\lambda_1, \lambda_2, \Gamma}(z) = \frac{p_{\lambda_1, \lambda_2, \Gamma}(z)}{(1 -z^{N_{\Gamma}})^{|\Phi^+| +1}}.$$ 
In particular, $F_{\lambda_1, \lambda_2, \Gamma}(z)$ is a complex rational function whose poles are the $N_{\Gamma}$-th roots of unity, each of order $|\Phi^+| +1$.} 
\end{proposition}

\begin{proof}
 By Thm. \ref{thm6}, we have 
\begin{equation}
\begin{split}
 F_{\lambda_1, \lambda_2, \Gamma}(z)
& = \sum\limits_{k \in \N}\ \sum\limits_{[y]} \text{vol} ( \Gamma_y \backslash G_y) \Psi_{\lambda_k}(y) z^{k-1}\\ 
& = \sum\limits_{[y]} \text{vol} (\Gamma_y \backslash G_y) \sum\limits_{k \in \N} \Psi_{\lambda_k} (y) z^{k-1},\\ 
\end{split}
\end{equation}

\noindent where $[y]$ runs over the $\Gamma$-conjugacy classes of elliptic elements in $\Gamma$. Now by definition of $\Psi_{\lambda_k}(y)$, we get 
\begin{equation}
\begin{split}
F_{\lambda_1, \lambda_2, \Gamma}(z)
& = \sum\limits_{[y]} \text{vol} (\Gamma_y \backslash G_y) \sum\limits_{k \in \N} \dfrac{ (-1)^{\tfrac{d + d_y}{2}} \prod\limits_{\alpha \in \Phi_y^+} \left(\langle \delta_y, \alpha \rangle \right) ^{-1} \Psi_{\lambda_k}^0(y)}{|W_y| [G_y : G_y^0]} z^{k-1}\\
& = \sum\limits_{[y]} \dfrac{\text{vol} (\Gamma_y \backslash G_y) (-1)^{\tfrac{d + d_y}{2}}}{|W_y| [G_y : G_y^0]} \sum\limits_{k \in \N} \prod\limits_{\alpha \in \Phi_y^+} \left(\langle \delta_y , \alpha \rangle\right) ^{-1} \Psi_{\lambda_k}^0(y) z^{k-1}.\\
\end{split}
\end{equation}

\noindent Let us denote $\dfrac{\text{vol}(\Gamma_y \backslash G_y) (-1)^{\frac{d+ d_y}{2}}}{|W_y| [ G_y : G_y^0]}$ by $C_y$. Finally, from the definition of $\Psi_{\lambda_k}^0(y)$, the generating function $F_{\lambda_1, \lambda_2, \Gamma}(z)$ equals

$$\sum\limits_{[y]} C_y \sum\limits_{k \in \N} \prod_{\alpha \in \Phi_y^+} (\langle \delta_y , \alpha \rangle) ^{-1} \dfrac{\sum\limits_{w \in W_c / W_y} (-1)^{l(w)} e ^{w \lambda_k -\delta}(y) \prod\limits_{\alpha \in \Phi_y^+} \langle w \lambda_k, \alpha \rangle}{\prod\limits_{\alpha \in \Phi^+ \setminus \Phi_y^+} (1 - e^{-\alpha}(y))}z^{k-1}.$$ 

\noindent We write this expression as 

$$F_{\lambda_1, \lambda_2, \Gamma}(z)= \sum\limits_{[y]} C_y \prod\limits_{\alpha \in \Phi^+ \setminus \Phi_y^+}(1 - e ^{-\alpha}(y))^{-1} \sum\limits_{w \in W_c / W_y} (-1)^{l(w)} a_{y, w}(z),$$

\noindent where $a_{y,w} (z) = \sum\limits_{k \in \N} e ^{w \lambda_k -\delta }(y) \prod\limits_{\alpha \in \Phi_y^+} \dfrac{\langle w \lambda_k , \alpha \rangle }{\langle \delta_y , \alpha \rangle} z^{k-1}$. 
 Therefore, it is enough to show that $a_{y,w}(z)$ is a rational function in $z$. 
 
\noindent We observe that
 \begin{equation}
 \begin{split}
 \prod\limits_{\alpha \in \Phi_y^+} \dfrac{\langle w \lambda_k , \alpha \rangle }{ \langle \delta_y, \alpha \rangle }
 & = \prod\limits_{\alpha \in \Phi_y^+} \dfrac{1}{\langle \delta_y, \alpha \rangle } \prod\limits_{\alpha \in \Phi_y^+} \langle w (\lambda_1 + k \lambda_2 ) , \alpha \rangle \\
 & = \prod\limits_{\alpha \in \Phi_y^+} \dfrac{1}{\langle \delta_y, \alpha \rangle} \prod\limits_{\alpha \in \Phi_y^+} ( \langle w \lambda_1, \alpha \rangle + \langle k w \lambda_2, \alpha \rangle )\\
 & = \prod\limits_{\alpha \in \Phi_y^+} \dfrac{1}{\langle \delta_y, \alpha \rangle} \prod\limits_{\alpha \in \Phi_y^+}( \langle w \lambda_1 , \alpha \rangle + k \langle w \lambda_2, \alpha \rangle ).\\
\end{split}
\end{equation} 

\noindent Therefore, $\prod\limits_{\alpha \in \Phi_y^+} \dfrac{\langle w \lambda_k , \alpha \rangle }{\langle \delta_y , \alpha \rangle}$ is a polynomial in $k$ of degree at most $|\Phi_y^+|$, and let us write it as $\sum\limits_{j=0}^{|\Phi_y^+|} b_j k^j$ for some complex numbers $b_j$.

\noindent Then $$a_{y,w}(z)=\sum\limits_{k \in \N} e^{w \lambda_1 - \delta}(y) e^{k w \lambda_2} (y) \sum\limits_{j=0}^{|\Phi_y^+|} b_j k^j z^{k-1}.$$

\noindent For the sake of simplicity, we show that $z . a_{y,w}(z)= e^{w \lambda_1-\delta}(y) \sum\limits_{j=0}^{|\Phi_y^+|} b_j \sum\limits_{k \in \N} k ^j (e^{w \lambda_2}(y) z )^k$ is a rational function in $z$.

\noindent We first claim that for any $0 \leq j \leq |\Phi_y^+|$, 
\vspace{2pt}
\begin{equation}\label{claim}
\begin{split}
\sum\limits_{k \in \N} k^j (e ^{w \lambda_2}(y) z )^k 
& = \dfrac{p_j(z)}{(1-e^{w \lambda_2}(y)z)^{|\Phi_y^+| +1}}
\end{split}
\end{equation}

\noindent for some polynomial $p_j(z)$ of degree at most $|\Phi_y^+|$.

\noindent We will show claim \ref{claim} by induction on $j$. Clearly, it is true for $j=0$. Assume that it is true for $\{1,2, \dots , j-1\}$. Observe that $j ! { k+j \choose k} = k^j + \sum\limits_{l=0}^{j-1} c_{\ell} k^j\ \text{for some}\ c_{\ell} \in \Z$. Using this we have, 
$$\sum\limits_{k \in \N}k^j x^k = \dfrac{j!}{(1-x)^{j+1}} - \sum\limits_{l=0}^{j-1}c_{\ell} \sum\limits_{k \in \N} k^{\ell} x^k.$$

\noindent Since $(1-x)^{-(j+1)}=\dfrac{(1-x)^{|\Phi^+_y|-j}}{(1-x)^{|\Phi^+_y|+1}}$, the claim \ref{claim} follows by substituting $x=e^{w \lambda_2}(y) z$.

\vspace{8pt} 

\noindent We write $$p_{y,w}(z)=e^{w \lambda_1 - \delta}(y)\sum\limits_{j=0}^{|\Phi^+_y|} b_j p_j (z).$$ 

\noindent Then by the claim \ref{claim}, we have 

$$z. a_{y,w}(z)=\frac{p_{y,w}(z)}{(1 - e^{w \lambda_2}(y) z)^{|\Phi_y^+|+1}}.$$

\noindent By the definition of $N_{\Gamma}$,  $e^{w \lambda_2}(y)$ is a $N_{\Gamma}$-th root of unity (not necessarily primitive). Therefore, we can write  

$$ \frac{p_{y,w}(z)}{(1 - e^{w \lambda_2}(y) z)^{|\Phi_y^+|+1}} = \frac{p_{y,w}(z) (1-z^{N_{\Gamma}})^{|\Phi^+|-|\Phi_y^+|} \prod\limits_{\xi^{N_{\Gamma}}=1, \xi \neq e^{w \lambda_2}(y)}(1-\xi z)^{|\Phi_y^+| +1}}{(1 - z^{N_{\Gamma}})^{|\Phi^+|+1}}.$$

\noindent Now $|\Phi^+| \geq |\Phi_y^+|$. Hence, the degree of the numerator of the above expression is less than or equal to $|\Phi_y^+| + N_{\Gamma} (|\Phi^+| - |\Phi_y^+| )+ (N_{\Gamma}-1) (|\Phi_y^+| +1) < N_{\Gamma} (|\Phi^+| +1).$
\end{proof}

\begin{remark}
\textit{The Prop. \ref{propgen} also implies that for a given string $\mathcal{S}(\lambda_1, \lambda_2)$ with $P^{\lambda_1}=P^{\lambda_2}$, and $\lambda_1$ or $\lambda_2$ belonging to $\mathcal{D}^{\star}(G)$, there are infinitely many discrete series representations $\pi_{\lambda}$ with Harish-Chandra parameters $\lambda \in \mathcal{S}(\lambda_1, \lambda_2)$ such that $\pi_{\lambda}$ occurs as a subrepresentation in $L^2(\Gamma \backslash G)$ with non-zero multiplicity.} 
\end{remark}

\begin{remark}
\textit{Note that for a given string, the generating function can be defined by simply taking the coefficient to be the Lefschetz number instead of the multiplicity of the discrete series. In that case, the proposition holds without assuming that the base or the direction of the string belongs to $\mathcal{D}^{\star}(G)$. Since we are concerned with only the multiplicity of discrete series, we need the Harish-Chandra parameters in $\mathcal{D}^{\star}(G)$ which is sufficient to relate the multiplicities and Lefschetz numbers up to sign (e.g. Thm. \ref{thm4}).}
\end{remark}

\section{Proof of the Main Results}

\subsection{Proof of Thm. \ref{thm1}}\label{proofthm1}

For simplicity, we write $N:=N_{\Gamma}$. We have the string $\mathcal{S}(\lambda_1, \lambda_2)$ with base $\lambda_1$, direction $\lambda_2$ and $P^{\lambda_1}=P^{\lambda_2}$; and by Prop. \ref{propgen} its generating function $F_{\lambda_1, \lambda_2, \Gamma }$ is of the form 
$$F_{\lambda_1, \lambda_2, \Gamma } (z) = \frac{p(z)}{(1-z^{N})^{|\Phi^+|+1}},$$ 
with deg $p(z) < N(|\Phi^+|+1)$. Let us write $p(z) = \sum\limits_{k=0}^{N(|\Phi^+|+1)-1} b_k z^k$ for some $b_k \in \C$. We claim that for any $0 \leq j \leq N-1$ and $m \geq0$, 

$$m(\pi_{\lambda_{mN+j}},\Gamma) = \sum\limits_{h=0}^{|\Phi^+|} b_{hN+j} \left( \begin{matrix} m-h+|\Phi^+| \\ |\Phi^+| \end{matrix} \right),$$ 

\noindent where $ \left( \begin{matrix} m-h+|\Phi^+| \\ |\Phi^+| \end{matrix} \right) =0\ \text{when}\ m<h$. Therefore, the sum essentially runs over $0 \leq h \leq \text{min}(|\Phi^+|,m)$. We use the following binomial expansion: $$\frac{1}{(1-y)^{j+1}} = \sum_{k \geq 0} \left( \begin{matrix} k+j \\ k \end{matrix} \right) y^k \hspace{10pt} \text{for any}\ j \in \N.$$ 

\noindent Then, we have 
\begin{equation*}
\begin{split}
 F_{\lambda_1,\lambda_2,\Gamma} (z) & = p(z) (1-z^{N})^{-(|\Phi^+|+1)} \\
& = \left(\sum_{j=0}^{N-1} \sum_{h=0}^{|\Phi^+|} b_{hN+j} z^{hN+j} \right) \left( \sum_{k \geq0} \left( \begin{matrix} k+ |\Phi^+| \\ |\Phi^+| \end{matrix} \right) z^{kN} \right) \\
& = \sum_{j=0}^{N-1} \sum_{k\geq0} \sum_{h=0}^{|\Phi^+|} b_{hN+j} \left( \begin{matrix} k+ |\Phi^+| \\ |\Phi^+| \end{matrix} \right) z^{(k+h)N +j} \\
& = \sum_{j=0}^{N-1} \sum_{m\geq0} \left( \sum_{h=0}^{|\Phi^+|} b_{hN+j} \left( \begin{matrix} m-h + |\Phi^+| \\ |\Phi^+| \end{matrix} \right) \right) z^{mN+j}.
\end{split}
\end{equation*}
Let us fix a $j$ with $0 \leq j \leq N-1$. Since $|\mathcal{A} \cap (j+N \Z)| \geq |\Phi^+|+1$, there are $m_0,\dots ,m_{|\Phi^+|} \in \Z^+$ such that $m_i N+j \in \mathcal{A}$ for all $0 \leq i \leq |\Phi^+|$.

\noindent Therefore, 
$$ m(\pi_{\lambda_{m_i N +j}},\Gamma) = \sum_{h=0}^{|\Phi^+|} b_{hN+j}  \left( \begin{matrix} m_i - h + |\Phi^+| \\ |\Phi^+| \end{matrix} \right)\hspace{15pt}  \text{with }\ 0 \leq i \leq |\Phi^+| $$ 
is a system of $|\Phi^+| +1$ many linear equations in $|\Phi^+| +1$ many variables $\left\{ b_{hN+j} \right\}_{h=0}^{|\Phi^+|}$. It is not difficult to check that the square matrix 

$$ \Biggl \{ { m_i - h + |\Phi^+ | \choose |\Phi^+| }\Biggr \}  _{0 ~\leq~ h,~ i ~\leq ~|\Phi^+|}$$ of order $|\Phi^+| + 1$ is non-singular. Hence, it has a unique solution. Consequently, for any $0 \leq h \leq |\Phi^+|$, $b_{hN+j}=\sum\limits_i n(hN+j, m_iN+j)\ m(\pi_{\lambda_{m_i N +j}},\Gamma)$ for some integers $n(hN+j, m_iN+j)$ independent of $\Gamma$. In other words, the solution $\left\{ b_{hN+j} \right\}_{h=0}^{|\Phi^+|}$ can be expressed linearly in terms of $\{ m(\pi_{\lambda_{m_i N +j}},\Gamma)$:\ $m_iN+j \in \mathcal{A}\} $.

\vspace{5pt}

\noindent Since the finite set $\mathcal{A}$ satisfies the condition $|\mathcal{A} \cap (j+N \Z)| \geq |\Phi^+|+1$ for all $j \in \left\{ 0,1, \dots ,N-1 \right\}$, we have  $ b_{\ell} = \sum\limits_{k \in \mathcal{A}} n(\ell, k)\  m(\pi_{\lambda_k},\Gamma)$ for all $0 \leq \ell \leq N(|\Phi^+|+1)$. This implies that the polynomial $p(z)$ and hence the generating function $F_{\lambda_1, \lambda_2, \Gamma}$ is determined by the same. Therefore for any $\ell \geq 0$, the discrete series multiplicity $m(\pi_{\lambda_{\ell}},\Gamma)=\sum\limits_{k \in \mathcal{A}} n(\ell, k)\ m(\pi_{\lambda_k},\Gamma)$ for some integers $n(\ell, k)$ independent of $\Gamma$. 

\begin{remark}
\textit{Thm. \ref{thm1} implies the set of multiplicities of infinitely many discrete series representations parametrized by the string $\mathcal{S}(\lambda_1 , \lambda_2)$ is determined by a finite subset of multiplicities with cardinality $N (|\Phi^+| +1) $. } 
\end{remark}

\subsection{Proof of Cor. \ref{cor2}}\label{proofcor2}
 
Let $\mathcal{A}= \left\{k \in \N \cup \{0\} : m(\pi_{\lambda_k},\Gamma_1)=m(\pi_{\lambda_k},\Gamma_2) \right\}$. We need to show that $\mathcal{A}$ satisfies 
$$|\mathcal{A} \cap (N \Z +j)| \geq |\Phi^+|+1 \hspace{10pt} \text{for all} \hspace{5pt} 0 \leq j \leq N-1.$$
By the hypothesis, we have $$ \liminf\limits_{t \rightarrow \infty} \frac{| \left\{ 0 \leq k \leq t : k \in \mathcal{A} \right\}|}{t} > \frac{N-1}{N}.$$
$ \text{This means, for every}~\epsilon >0 ~\text{there is}~ r_0 >0~ \text{such that}~$  $$\frac{|\left\{ k \in \mathcal{A} : k<rN(|\Phi^+|+1) \right\}|}{rN(|\Phi^+|+1)} > \frac{N-1+\epsilon}{N} \hspace{10pt} \text{for all}~r \geq r_0.$$ 
Equivalently,
$$| \left\{ k \in \mathcal{A} : k<rN(|\Phi^+|+1) \right\}| > r(N-1+\epsilon) (|\Phi^+|+1).$$
Let $r \geq r_0$. Fix a $j_0$ such that $0 \leq j_0 \leq N-1$.
Then the cardinality
\begin{equation*}
\begin{split}
| \left\{ k \in \mathcal{A} : k<rN(|\Phi^+|+1) \right\}| &= \sum_{j=0}^{N-1} |\left\{ k \in \mathcal{A} : k<rN(|\Phi^+|+1) \right\} \cap (j+N \Z)|\\
&\leq |\mathcal{A} \cap (j_0 + N\Z)| + (N-1)r (|\Phi^+|+1).
\end{split}
\end{equation*}
In other words, $$|\mathcal{A} \cap (j_0 + N\Z)|  > r \epsilon (|\Phi^+|+1).$$
This holds for all $r \geq r_0$. In particular for $r \geq \frac{1}{\epsilon}$, we have $\mathcal{A} \cap (j_0+ N\Z) \geq (|\Phi^+|+1)$.
This holds for all $j_0 \in \left\{ 0,1, \dots ,N-1 \right\}$.

\subsection{Exhaustion of Harish-Chandra parameters in $\mathcal{D}^{\star}(G)$}\label{exhaustion}
In this subsection, we describe the notion of a multi-directed string of Harish-Chandra parameters and prove analogous result to Thm. \ref{thm1} in this generality. Let us denote the cartesian product $\N \times \N \times \dots \times \N$ of $n$ copies of $\N$ by $\N^n$.

\begin{definition}\label{multi-directed}
Given a finite set of linear forms $\lambda_1, \lambda_2,\dots ,\lambda_n \in \mathcal{D}(G)$, the infinite set $\mathcal{S}(\lambda_1, \lambda_2, \dots, \lambda_n)$ $:=\{ k_1 \lambda_1 + k_2 \lambda_2 + \dots + k_n \lambda_n : (k_1, k_2, \dots ,k_n) \in \N^n \setminus \{(0,0, \dots ,0)\} \}$ is called a multi-directed string with directions $\lambda_1, \lambda_2, \dots, \lambda_n$. For every $\boldsymbol{k}:= (k_1, k_2, \dots ,k_n) \in \N^n \setminus \{(0,0, \dots ,0)\} $, the linear form $\sum\limits_{i=1}^{n} k_i \lambda_i$ is denoted by $\lambda_{\boldsymbol{k}}$. 
\end{definition}

\noindent It is not difficult to see the following generalisation of Lem. \ref{lemma2.4.3}: 
\vspace{2pt}
\begin{lemma}\label{lem-grid}

\textit{Assume $\lambda_1, \lambda_2, \dots , \lambda_n \in \mathcal{D}(G)$ such that $P^{\lambda_r}=P^{\lambda_s}$ for all $1 \leq r, s \leq n$, and $\lambda_1 \in \mathcal{D}^{\star}(G)$. Then every linear form in the multi-directed string $\mathcal{S}(\lambda_1, \lambda_2, \dots ,\lambda_n)$ belongs to $\mathcal{D}^{\star}(G)$.}

\end{lemma}

  Now we can generalise the Thm. \ref{thm1} for a multi-directed string. First, we need a few notations: Let $p_i : \N^n \longrightarrow \N$ be the projection onto $i$-th component of an $n$-tuple of natural numbers. For a subset $\boldsymbol{\mathcal{A}}$ of $\N^n$, let $\mathcal{A}_i$ denote the image of $\boldsymbol{\mathcal{A}}$ under $p_i$ for $1 \leq i \leq n$.

\begin{theorem}\label{cor-n-grid}
\textit{Assume the hypothesis on $G$ and $K$ as in Thm. \ref{thm1}. Let $N_{\Gamma}$ be the maximal order of elliptic elements in $\Gamma$. For given $\lambda_1, \lambda_2, \dots ,\lambda_n \in \mathcal{D}(G)$ such that $P^{\lambda_r}=P^{\lambda_s}$ for all $1 \leq r, s \leq n$ and at least one of the $\lambda_r$'s belongs to $\mathcal{D}^{\star}(G)$, consider the multi-directed string $\mathcal{S}(\lambda_1, \lambda_2, \dots ,\lambda_n)$. If $\boldsymbol{\mathcal{A}}$ is a finite subset of $\N^n$ such that each $\mathcal{A}_i$ satisfies $| \mathcal{A}_i  \cap (j + N_{\Gamma} \Z) | \geq |\Phi^+|+1$ for all $ 0 \leq j \leq N_{\Gamma}-1$, then for any  $\lambda_{\boldsymbol{k}} \in \mathcal{S}(\lambda_1, \lambda_2, \dots ,\lambda_n)$ the discrete series multiplicity $m(\pi_{\lambda_{\boldsymbol{k}}}, \Gamma)$ can be expressed linearly in terms of the finitely many discrete series multiplicities $\{ m(\pi_{\lambda_{\boldsymbol{k}}}, \Gamma) : \boldsymbol{k} \in \boldsymbol{\mathcal{A}\}}$.} 
\end{theorem}

\begin{proof}

By Lem. \ref{lem-grid}, we have $\lambda_{\boldsymbol{k}} \in \mathcal{D}^{\star}(G)$ for every $\boldsymbol{k} \in \N^n \setminus \{(0,0, \dots ,0)\}$. By induction on $n$, we are reduced to the case for $n=2$. Let us fix $(t_1, t_2) \in \N^2$. We need to show that $m(\pi_{t_1 \lambda_1 + t_2 \lambda_2}, \Gamma)$ can be linearly expressed in terms of the finite set $\{ m(\pi_{\lambda_{\boldsymbol{k}}}, \Gamma) : \boldsymbol{k} \in \boldsymbol{\mathcal{A}}\}$. We have the string $\mathcal{S}(t_1 \lambda_1, \lambda_2)$ with base $t_1 \lambda_1$ and direction $\lambda_2$. Since $\mathcal{A}_2$ satisfies $| \mathcal{A}_2  \cap (j + N_{\Gamma} \Z) | \geq |\Phi^+|+1$ for all $ 0 \leq j \leq N_{\Gamma}-1$, by Thm. \ref{thm1} the multiplicity $m(\pi_{t_1 \lambda_1 + t_2 \lambda_2}, \Gamma)$ can be expressed linearly in terms of the finite set $\{ m(\pi_{\lambda_{\boldsymbol{k}}}, \Gamma) : p_1 (\boldsymbol{k})=t_1 , p_2 (\boldsymbol{k}) \in \mathcal{A}_2 \}.$

\noindent For any $\ell \in \mathcal{A}_2$, we have the string $\mathcal{S}(\ell \lambda_2, \lambda_1)$ with base $ \ell \lambda_2$ and direction $\lambda_1$. Since $\mathcal{A}_1$ satisfies $| \mathcal{A}_1  \cap (j + N_{\Gamma} \Z) | \geq |\Phi^+|+1$, again by Thm. \ref{thm1} the multiplicity $m(\pi_{ \ell \lambda_2 + t_1 \lambda_1}, \Gamma)$ can be expressed linearly in terms of the finite set $\{ m(\pi_{\lambda_{\boldsymbol{k}}}, \Gamma) : p_1 (\boldsymbol{k}) \in \mathcal{A}_1, p_2 (\boldsymbol{k})=\ell \}$. This holds for every $\ell \in \mathcal{A}_2$. Therefore, $m(\pi_{t_1 \lambda_1 + t_2 \lambda_2}, \Gamma)$ can be expressed linearly in terms of the finite set $\{ m(\pi_{\lambda_{\boldsymbol{k}}}, \Gamma) : p_1(\boldsymbol{k}) \in \mathcal{A}_1, p_2 (\boldsymbol{k}) \in \mathcal{A}_2\}$.
\end{proof}

\noindent It is tempting to ask whether there exists a finite set $\{\lambda_1, \lambda_2, \dots ,\lambda_n\} \subset \mathcal{D}(G)$, with the property $P^{\lambda_r}=P^{\lambda_s}$ for all $1 \leq r,s \leq n$ and at least one of the $\lambda_r$'s belongs to $\mathcal{D}^{\star}(G)$, so that the set quality $\{ \mu \in \mathcal{D}^{\star}(G): P^{\mu}=P^{\lambda_r} \}=\mathcal{S}(\lambda_1, \lambda_2, \dots ,\lambda_n)$ holds. The inclusion $\mathcal{S}(\lambda_1, \lambda_2, \dots, \lambda_n) \subset \{ \mu \in \mathcal{D}^{\star}(G) : P^{\mu}=P^{\lambda_r} \}$ is obvious, whereas the reverse containment is not true by the following proposition: 

 \begin{proposition}\label{prop-exhaustion}
\textit{Let $\{ \lambda_1, \lambda_2, \dots , \lambda_n\}$ be a $\Z$-linearly independent subset of $\mathcal{D}(G)$ with at least one of the $\lambda_r$'s belonging to $\mathcal{D}^{\star}(G)$ and $P^{\lambda_r}=P^{\lambda_s}$ for all $1 \leq r,s \leq n$. Then there exist infinitely many $\mu \in \mathcal{D}^{\star}(G)$ with $P^{\mu}=P^{\lambda_r}$ such that $\mu \notin \mathcal{S}(\lambda_1, \lambda_2, \dots, \lambda_n)$.} 
\end{proposition}

\begin{proof}

Without loss of generality, let us assume that $\lambda_1 \in \mathcal{D^*}(G)$. We show that $m \lambda_1 - \lambda_2 \in \mathcal{D^*}(G)$ for sufficiently large positive integer $m$. Clearly $\Phi = P^{\lambda_1} \cup P^{-\lambda_1} = P^{\lambda_2} \cup P^{-\lambda_2}$. For any $\alpha \in P^{\lambda_1}=P^{\lambda_2}$, we have $\langle \lambda_1 , \alpha \rangle > 0$ and $\langle \lambda_2 , \alpha \rangle > 0$. So, there exists $k \in \N$ such that for each $\alpha \in P^{\lambda_1}=P^{\lambda_2}$, $m \langle \lambda_1 , \alpha \rangle > \langle \lambda_2 , \alpha \rangle$ for all $m \geq k$. Equivalently, for every $\alpha \in P^{-\lambda_1} =P^{-\lambda_2}$, $m \langle \lambda_1 , \alpha \rangle < \langle \lambda_2 , \alpha \rangle$. In other words, if $\alpha \in \Phi$ then for every $m \geq k$, we get $\langle m \lambda_1 - \lambda_2, \alpha \rangle > 0$ if and only if $\alpha \in P^{\lambda_1}=P^{\lambda_2}$. Therefore, $P^{m \lambda_1 - \lambda_2}=P^{\lambda_1}=P^{\lambda_2}$. For a fixed $m > k$, let us choose $\alpha \in P^{m \lambda_1 - \lambda_2} \cap \Phi_n$. Then we have $\langle m \lambda_1 - \lambda_2 - \delta^{m \lambda_1 - \lambda_2} , \alpha \rangle = \langle m \lambda_1 - \delta^{\lambda_1}, \alpha \rangle - \langle \lambda_2 , \alpha \rangle = \langle \lambda_1 - \delta^{\lambda_1} , \alpha \rangle + \langle (m-1) \lambda_1 - \lambda_2 , \alpha \rangle$. Now the first summand is positive because $\lambda_1 \in \mathcal{D^*}(G)$, and the second summand is positive by the choice of $m$. Therefore, $m \lambda_1 - \lambda_2 \in \mathcal{D^*}(G)$ for all $m >k$. Finally, by the $\Z$-linear independence of $\lambda_i$'s, the linear form $\mu:= m \lambda_1 - \lambda_2 \notin \mathcal{S}(\lambda_1, \lambda_2, \dots, \lambda_n)$. Hence, the proposition follows. 
\end{proof}

\begin{remark}
\textit{From the above proof it can be seen easily that for any $1 < i \leq n$, the linear form $m \lambda_1 - \lambda_i$ with sufficiently large $m$   provides an infinite collection of such $\mu$ which determines the same positive system but does not belong to the multi-string. The Prop. \ref{prop-exhaustion} also holds for any arbitrary finite set of $n$ linear forms in which any subset of $n-1$ linear forms is $\Z$-linearly independent.}
\end{remark}

\section{Dimensions of the space of Cusp forms}\label{sec-5}

In this section, we showcase an application of Thm. \ref{thm1} to the dimension of the space of cusp forms. We need to write the explicit dictionary between the dimension of the space of cusp forms and discrete series multiplicities. 

Let $B$ be a division quaternion algebra over a totally real number field $F$ of degree $n$. Assume that $B$ splits at exactly one real place $v$ of $F$ i.e. $B_v \cong \mathrm{M}(2, \R)$. Let $|.|$ be the standard norm on $B$. Let $B^1 = \{ x \in B :  |x|=1 \}$. Then $B^1_v \cong \mathrm{SL}(2, \R)$. Consider a maximal order $\mathcal{O}$ in $B$. Let $\mathcal{O}^1 = \{ x \in \mathcal{O} : |x| =1 \}$.  Let $\mathcal{H}$ denote the Hamiltonian algebra over $\R$. We have the map 
$$ \iota : B^1 \longrightarrow B_{\R}^1 \cong (\mathrm{SL}(2, \R) \otimes \mathcal{H}^{n-1}) \longrightarrow \mathrm{SL}(2, \R),$$
where the first map is the usual inclusion and the second map is the projection onto the first component. 
Let $\Gamma$ be the image of $\mathcal{O}^1$ under the map $\iota$ in $\mathrm{SL}(2, \R)$. Then $\Gamma$ is a uniform lattice in $\mathrm{SL}(2, \R)$. In fact, every uniform lattice in $\mathrm{SL}(2, \R)$ arises in this way (see \cite{AW}). Let us recall the definition of cusp forms. 

\begin{definition}
\textit{Let $\Gamma \subset \mathrm{SL}(2, \R)$ be a uniform lattice and fix an integer $k$. A holomorphic function $f : \mathbb{H} \longrightarrow \C$ is called a cusp form of weight $k$ and level $\Gamma$ if} 

1. \textit{$f\left(\dfrac{az+b}{cz+d}\right)=(cz+d)^k f(z) \hspace{10pt}\ \text{for all}\  \begin{pmatrix}
a & b\\
c & d
\end{pmatrix} \in \Gamma,~ z \in \mathbb{H}$};

2. \textit{$\int\limits_{\Gamma \backslash \mathbb{H}} | f(z) |^2 y^k \dfrac{dx dy}{y^2} < \infty.$}
\end{definition}

\noindent The space of cusp forms of weight $k$ for a uniform lattice $\Gamma \subset \mathrm{SL}(2, \R)$ is denoted by $\mathcal{S}_k(\Gamma)$. For each $f \in \mathcal{S}_k(\Gamma)$, we define $\phi_f : \mathrm{SL}(2, \R) \longrightarrow \C$ by 
$$ \phi_f (g) = (ci + d)^{-k} f\left(\dfrac{a i + b}{c i +d}\right)\ \text{for all}\ g=\begin{pmatrix}
a & b\\
c & d
\end{pmatrix} \in \mathrm{SL}(2, \R),$$
where $i= \sqrt{-1}$. 
It can be seen easily that $\phi_f (\gamma g)= \phi_f (g)$ for all $ \gamma \in \Gamma$ and $g \in \mathrm{SL}(2, \R)$. Also, for any $s_{\theta}:=\begin{pmatrix}
\cos \theta & -\sin \theta\\
\sin \theta & \cos \theta
\end{pmatrix} \in \mathrm{SO}(2)$ and $g \in \mathrm{SL}(2, \R)$, $\phi_f (g s_{\theta})= e^{-i k \theta} \phi_f (g)$. Moreover, $\phi_f \in L^2(\Gamma \backslash \mathrm{SL}(2, \R))$. Therefore, we have the following linear map:
$$ T : \mathcal{S}_k(\Gamma) \longrightarrow L^2(\Gamma \backslash \mathrm{SL}(2, \R))\ \text{defined by}\ T(f)=\phi_f.$$

Let $\mathfrak{g}=\mathfrak{sl}(2, \C)$ be the complexified Lie algebra of  $\mathrm{SL}(2, \R)$. The Cartan subalgebra $\mathfrak{t}$ equals $\C X$, where $X=\begin{pmatrix}
    0 & -1\\
    1 & 0
\end{pmatrix}$. The set $\Phi$ of roots is $\{\pm \alpha\}$, where $\alpha$ is defined by $\alpha(X)=2i$. We fix the positive system $\Phi^+=\{\alpha \}$.

\vspace{5pt}

\noindent Any $\lambda \in \mathcal{D}(\mathrm{SL}(2, \R))$ is of the form $\lambda=\tfrac{k-1}{2}\alpha$ for some $k \in \Z$ with $|k| \geq 2$. Let $\pi_k$ be the discrete series representation associated to the Harish-Chandra parameter $\lambda=\tfrac{k-1}{2}$. If $k \geq -2$, then $\pi_k$ is called a holomorphic discrete series representation of $\mathrm{SL}(2, \R)$.

\begin{proposition}\label{prop-dic}
\textit{Let $\Gamma$ be a uniform lattice in $\mathrm{SL}(2, \R)$. Then $\mathcal{S}_k(\Gamma) \cong \Hom_G(\pi_{-k} , L^2(\Gamma \backslash \mathrm{SL}(2, \R) ))$ for all $k \geq 2$.}
\end{proposition}
\begin{proof}
For any $k \in \Z$, let $\tau_k$ be the character of $\mathrm{SO}(2)$ defined by $\tau_k(s_{\theta})=s_{\theta}^k$. It is well known that for every $f \in \mathcal{S}_k(\Gamma)$, the image $T(f)=\phi_f$ in $L^2(\Gamma \backslash \mathrm{SL}(2, \R))$ generates an irreducible subrepresentation of $L^2(\Gamma \backslash \mathrm{SL}(2, \R))$, which is isomorphic to the holomorphic discrete series $\pi_{-k}$ of $\mathrm{SL}(2, \R)$. Under the action of $\mathrm{SO}(2)$, $\phi_f$ transforms via the character $\tau_{-k}$. Since $\tau_{-k}$ is the minimal $\mathrm{SO}(2)$-type of the holomorphic discrete series representation $\pi_{-k}$ of $\mathrm{SL}(2, \R)$, $\phi_f$ belongs to the one-dimensional $\tau_{-k}$-isotypic component of this copy of $\pi_{-k}$ in $L^2(\Gamma \backslash \mathrm{SL}(2, \R))$. Therefore, if $f_1$ and $f_2$ are two linearly independent cusp forms in $\mathcal{S}_k(\Gamma)$ then the irreducible subrepresentations $\langle \phi_{f_1} \rangle$ and $\langle \phi_{f_2} \rangle$ in $L^2(\Gamma \backslash \mathrm{SL}(2, \R))$ generated by $\phi_{f_1}$ and $\phi_{f_2}$ respectively, are linearly independent; in fact, those are orthogonal to each other. In other words, the linear map $T$ is injective and it induces the isomorphism between $\mathcal{S}_k(\Gamma)$ and $\Hom_G(\pi_{-k} , L^2(\Gamma \backslash \mathrm{SL}(2, \R) ))$.     
\end{proof}

\begin{remark}
   \textit{The reader is referred to Thm. 2.10 of \cite{SG} for a more detailed proof.}
\end{remark}

\noindent In particular, from Prop. \ref{prop-dic} we have $\text{dim}\ \mathcal{S}_k(\Gamma) = m(\pi_{-k} , \Gamma)$, where $m(\pi_{-k} , \Gamma)$ is the multiplicity of $\pi_{-k}$ in $  L^2(\Gamma \backslash \mathrm{SL}(2, \R))$. Using this, we deduce the following consequence of Thm. \ref{thm1}. 

\begin{corollary}\label{cor-dim}
\textit{Let $\Gamma$ be a uniform lattice in $\mathrm{SL}(2, \R)$, and let $N_{\Gamma}$ be the maximal order of elliptic elements in $\Gamma$. Let $\mathcal{A}$ be a finite subset of $\N$, which satisfies $| \mathcal{A} \cap (j + N_{\Gamma} \Z) | \geq 2$ for all $0 \leq j \leq N_{\Gamma}-1$. Then for any $\ell \geq 2$, $\text{dim}\ \mathcal{S}_{\ell}(\Gamma)=\sum\limits_{k} n(\ell, k)\ \text{dim}\ \mathcal{S}_k(\Gamma)$, where the summation runs over the set $\{k \geq 2  :  \tfrac{k}{2} \in \mathcal{A}\ \text{or}\ \tfrac{k+1}{2} \in \mathcal{A}\}$ and $n(\ell, k)$ are integers which are independent of $\Gamma$.}
\end{corollary}
\begin{proof}
    By Prop. \ref{prop-dic}, the set $\{ \text{dim}\ \mathcal{S}_k(\Gamma) : k \geq 2\}$ equals to the set $\{ m(\pi_{-k}, \Gamma) : k \geq 2 \}$. The corresponding set of Harish-Chandra parameters is $\{ - \tfrac{k+1}{2}\alpha : k \geq 2\}$. Note that $\{ - \tfrac{k+1}{2}\alpha : k \geq 2\}= \{ -\frac{\alpha}{2} + m (-\alpha) : m \geq 1\} \cup \{ n (-\alpha): n \geq 2\}$. Moreover, it is easy to check that $-\alpha \in \mathcal{D^\star}(\mathrm{SL}(2, \R))$. Therefore, the corollary follows from Thm. \ref{thm1}.  
\end{proof}

\end{document}